\documentclass[10pt,a4paper]{amsart}

\usepackage{amssymb,amsmath,amsthm} 
\usepackage{rotating,graphicx,psfrag,epsfig}
\usepackage{cite}

\theoremstyle{plain}
\newtheorem{lemma}{Lemma}
\newtheorem{theorem}{Theorem}

\theoremstyle{definition}
\newtheorem{definition}{Definition}

\newtheorem{remark}{Remark}

\newcommand{\Z}{\mathbb{Z}}

\newcommand{\N}{\mathbb{N}}
\newcommand{\R}{\mathbb{R}}
\renewcommand{\H}{\ensuremath{\mathfrak{H}}}
\DeclareMathOperator*{\Res}{Res}
\renewcommand{\P}{\mathbb{P}}

\newcommand{\E}{\mathrm{\mathbb E}}

\newcommand{\floor}[1]{\ensuremath{\left\lfloor #1 \right\rfloor}}
\newcommand{\ceil}[1]{\ensuremath{\left\lceil #1 \right\rceil}}
\renewcommand{\vec}[1]{ \ensuremath{ {\bf#1}} }

\title[The height and range of Watermelons without wall]{The height and range of Watermelons without wall}
\author[Thomas Feierl]{Thomas Feierl$^\ddagger$}
\address{Thomas Feierl \\ Fakult\"at f\"ur Mathematik \\ Universit\"at Wien \\ Nordbergstr. 15 \\ 1090 Wien \\ Austria}
\thanks{$^\ddagger$ Research supported by the Austrian Science Foundation FWF, grant S9607-N13}
\date{\today}

\begin{document}
\begin{abstract}
		We determine the weak limit of the distribution of the random variables ``height'' and ``range'' on the set
		of $p$-watermelons without wall restriction as the number of steps tends to infinity.
		Additionally, we provide asymptotics for the moments of the random variable ``height''.
\end{abstract}
\maketitle

\section{Introduction}
%
%
The model of \emph{vicious walkers} was originally introduced by Fisher~\cite{MR751710} as a model
for wetting and melting processes. In general, the vicious walkers model is concerned with $p$ random
walkers on a $d$-dimensional lattice. In the lock step model, at each time step all of the walkers
move one step in one of the allowed directions, such that at no time any two random walkers share the 
same lattice point.

A configuration that attracted much interest amongst mathematical physicists and combinatorialists is
the \emph{watermelon configuration}, which is the model underlying this paper (see Figure~\ref{fig:wm:nowall}
for an example). This configuration can be studied with or without the presence of an impenetrable wall.
By tracing the paths of the vicious walkers through the lattice we can identify the (probabilistic)
vicious walkers model with certain sets of non-intersecting lattice paths. It is exactly this
equivalent point of view that we adopt in this paper.
We proceed with a precise definition. A \emph{$p$-watermelon of length $2n$} is a set of $p$ lattice
paths in $\Z^2$ satisfying the following conditions:
\begin{itemize}
		\item the $i$-th path starts at position $(0,2i)$ and ends at $(2n,2i)$,
		\item the paths consist of steps from the set $\left\{ (1,1),(1,-1) \right\}$ only and
		\item the paths are \emph{non-intersecting}, that is, at no time any two path share the same lattice point.
\end{itemize}
An example of a $4$-watermelon of length $16$ is shown in Figure~\ref{fig:wm:nowall} (for the moment, the dashed lines and the
labels should be ignored).
\begin{figure}[width=8cm]
		\begin{center}
				\includegraphics{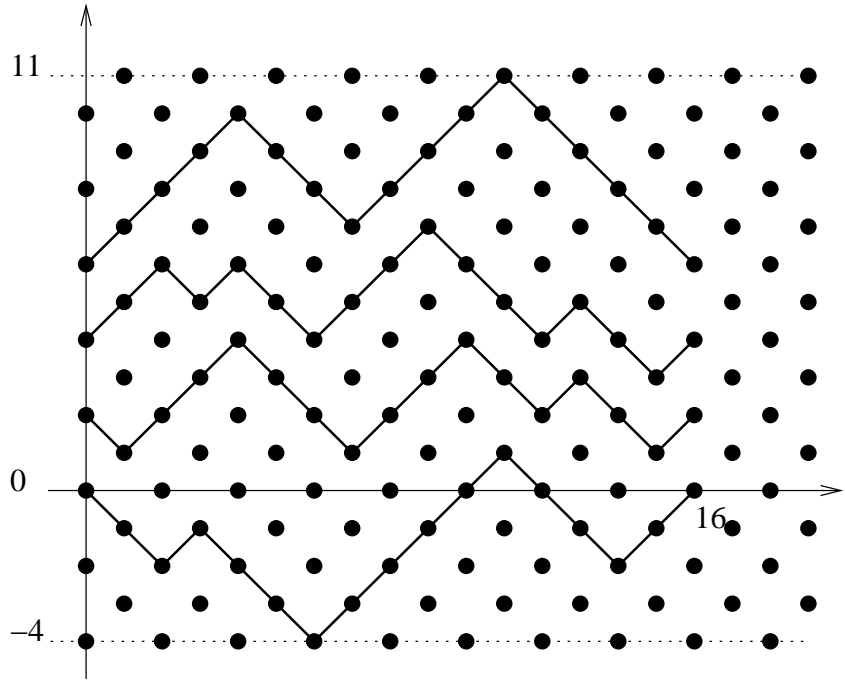}
		\end{center}
		\caption{Example of a $4$-watermelon of length $16$ without wall, height $11$, depth $4$ and range $15$}
		\label{fig:wm:nowall}
\end{figure}

Since its introduction, the vicious walkers model has been studied in numerous papers.
While early results mostly analyse the vicious walkers model in the continuum limit, there are nowadays
many results for certain configurations  directly based on the lattice path description given above.
With the increasing number of results it became clear that vicious walkers are very important objects
in mathematical areas far beyond its original scope.
For example, Guttmann, Owczarek and Viennot~\cite{MR1651492} related the star and watermelon
configurations to the theory of Young tableaux and integer partitions.
Later, Krattenthaler, Guttmann and Viennot~\cite{MR1801472} proved new, exact as well as asymptotic,
results for the number of certain configurations of vicious walkers.

The vicious walkers model is also very closely related to random matrix theory, as can be seen 
from articles by, e.g.,  Baik~\cite{MR1773413}, Johansson~\cite{MR1900323} and Nagao and Forrester~\cite{MR1877962}.
Later, Katori and Tanemura~\cite{MR2029612} and Gillet~\cite{gillet} studied the diffusion scaling
limit of certain configurations of vicious walkers, namely stars and watermelons, respectively.
%

In 2003, Bonichon and Mosbah~\cite{MR2022577} presented an algorithm for uniform random generation of
watermelons, which relies on the counting results by Krattenthaler, Guttmann and Viennot~\cite{MR1801472}.
Amongst other things, Bonichon and Mosbah studied the parameter height on the set of watermelons (with and without wall).

In this paper we rigorously analyse the following two parameters on the set of $p$-watermelons:
\begin{itemize}
		\item The \emph{height} of a watermelon is the maximum ordinate reached by its top most branch.
		\item The \emph{range} of a watermelon is the difference of the maximum of its top most branch and
				the minimum of its bottom most branch (the \emph{depth} of the watermelon).
\end{itemize}
The $4$-watermelon depicted in Figure~\ref{fig:wm:nowall} has the height $11$ and the range $11+4=15$.

Katori et. al.~\cite{katori2} and Schehr et. al.~\cite{schehr} studied the parameter ``height'' in the continuous limit,
and recovered the leading terms for some of the asymptotics proved in
\cite{watermelons:withwall,watermelons:withoutwall}. Additionally, Schehr et. al. gave some arguments concerning
the behaviour of the parameter ``height'' as the number of walkers tends to infinity.

Now, consider the set $\mathfrak m_{n}^{(p)}$ of $p$-watermelons of length $2n$, endowed with the uniform probability measure.
We can then speak of the random variables ``height'', denoted by $H_{n,p}$, and ``range'', denoted by $R_{n,p}$, on this set.
We determine the weak limits of $H_{n,p}$ and $R_{n,p}$ as the number $n$ of steps tends to infinity (see Theorem~\ref{thm:height:weak_limit}
and Theorem~\ref{thm:range_weaklimit}, respectively).
Additionally, we determine asymptotics for the moments of $H_{n,p}$.
In particular, we prove that the $s$-th moment of the random variable ``height'' behaves
like $\kappa_sn^{s/2}+\tau_sn^{(s-1)/2}+O\left( n^{s/2-1} \right)$ for some explicit numbers $\kappa_s$ and
$\tau_s$, see Theorem~\ref{thm:height_moments}.

Techniques similar to those applied in this paper can also be used to analyse the random variable height
on the set of $p$-watermelons under the presence of an impenetrable wall.
For details we refer to~\cite{watermelons:withwall}.

The paper is organised as follows. The next section contains some well known results that are
needed in the subsequent sections. In Section~\ref{sec:height} we consider the random variable ``height'',
and we determine the weak limit as well as asymptotics for all moments.
In the last section, we determine the weak limit of the random variable ``range''.

\section{Preliminaries}\label{sec:preliminaries}
In this section we collect several results which will be needed in the two subsequent sections. All these
results are either well known in the literature and/or can easily be derived by means of standard techniques. We,
therefore, remain very brief, give only a few comments on the proofs and in each case refer to the corresponding literature for details.

We start with an exact enumeration result for the total number of watermelons confined to a horizontal strip.
\begin{lemma}
		The number $m_{n,h,k}^{(p)}$ of $p$-watermelons without wall, length $2n$, height $<h$ and depth $>-k$ is given by
		\begin{align*}
			m_{n,h,k}^{(p)} 
				&=  \det_{0\le i,j<p}\left( \sum_{\ell\in\Z}\left( \binom{2n}{n+\ell(h+k)+i-j}-\binom{2n}{n+\ell(h+k)+h-i-j} \right) \right).
		\end{align*}
		The total number $m_{n}^{(p)}$ of $p$-watermelons is given by
		\begin{align*}
			m_{n}^{(p)} &= \det_{0\le i,j<p}\left( \binom{2n}{n+i-j} \right).
		\end{align*}
\end{lemma}
This lemma follows immediately from the well-known Lindstr\"om--Gessel--Viennot formula 
(see \cite[Corollary 3]{gesselviennot} or \cite[Lemma 1]{MR0335313}), together with an iterated 
reflection principle.

\begin{remark}
		Since any $p$-watermelon without wall and length $2n$ has depth $>-n-1$, we see that the number of watermelons with height $<h$
		and no restriction on the depth is given by $m_{n,h,n+1}^{(p)}$. For the sake of convenience, this quantity will also be denoted by
		$m_{n,h}^{(p)}$. In this special case, the determinantal expression above simplifies to
		\begin{align*}
				m_{n,h}^{(p)} &= \det_{0\le i,j<p}\left( \binom{2n}{n+i-j}-\binom{2n}{n+h-i-j} \right).
		\end{align*}
\end{remark}
\begin{lemma}\label{lem:total_num_asymptotics} 
We have
\begin{equation*}
m_{n}^{(p)} 
= \left(\frac{2}{n}\right)^{\binom{p}{2}}\binom{2n}{n}^p\left(\prod_{i=0}^{p-1}i!\right) \left( 1+O(n^{-1}) \right)
\end{equation*}
as $n\to\infty$.
\end{lemma}
\begin{proof}[Proof (Sketch)]
		The result is established from the closed form expression for $m_{n}^{(p)}$, viz
		\[
		m_{n}^{(p)} = \det_{0\le i,j<p}\left( \binom{2n}{n+i-j} \right)
		= \binom{2n}{n}^p\left(\prod_{i=0}^{p-1}i!\frac{(2n+i)!}{(2n)!}\left( \frac{n!}{(n+i)!} \right)^2\right).
		\]
		For details on the evaluation of this (and many more) determinant, we refer to \cite{MR1701596}.
\end{proof}
\begin{lemma}
For $|m-z|\le n^{5/8}$, $z$ bounded, and arbitrary $N>1$ we have the asymptotic expansion
		\begin{multline} \label{eq:approx_binom_quot}
				\frac{\binom{2n}{n+m-z}}{\binom{2n}{n}}
				= e^{-m^2/n}
				\sum_{u=0}^{4N+1}\left(\frac{z}{\sqrt n} \right)^{u}\frac{1}{u!}H_u\left( \frac{m}{\sqrt n} \right)\\
				+e^{-m^2/n}\sum_{u=0}^{4N+1}\left(\frac{z}{\sqrt n} \right)^{u}
					\sum_{l=1}^{3N+1}n^{-l}\sum_{k=0}^{u-1}
					\sum_{r=1}^{2l}F_{r,l}\binom{2r}{u-k}\frac{(-1)^{u-k}}{k!}H_k\left( \frac{m}{\sqrt n}  \right)
					\left( \frac{m}{\sqrt n}  \right)^{2r+k-u} \\
					+O\left( e^{-m^2/n}n^{-1-2N} \right)
		\end{multline}
		as $n\to\infty$. 
		Here, the $F_{r,l}$ are some constants the explicit form of which is of no importance in the sequel, and 
		$H_k(z)$ denotes the $k$-th Hermite polynomial, that is,
		\begin{align}
				\frac{H_k(z)}{k!} &= \sum_{m\ge 0}\frac{(-1)^{k-m}}{(k-m)!}\frac{(2z)^{2m-k}}{(2m-k)!},\qquad k\ge 0.
				\label{eq:hermitepolynomial}
 		\end{align}
		\label{lem:approx_binom_quot}
\end{lemma}
The lemma above follows from Stirling's approximation for the factorials.
For a detailed proof we refer to \cite[Lemma 6]{watermelons:withwall}.

\section{Height}\label{sec:height}
In this section we derive asymptotics for the distribution as well as for the moments of the random variable $H_{n,p}$.
As mentioned before, the number of $p$-watermelons with length $2n$ and height $<h$ is given by $m_{n,h}^{(p)}=m_{n,h,n+1}^{(p)}$.
Consequently, we have for the distribution of $H_{n,p}$
\begin{equation}\label{eq:height_cdf}
\P\left\{ H_{n,p}+1 \le h \right\} = \frac{m_{n,h}^{(p)}}{m_{n}^{(p)}}.
\end{equation}
\begin{theorem}\label{thm:height:weak_limit}
		For each fixed $t\in(0,\infty)$ we have the asymptotics
		\begin{equation}
				\P\left\{ \frac{H_{n,p}+1}{\sqrt n}\le t \right\}=
				\frac{2^{-\binom{p}{2}}}{\prod_{j=0}^{p-1}j!}
				\det_{0\le i,j<p}\left( (-1)^iH_{i+j}(0)-H_{i+j}\left( t \right)e^{-t^2} \right)
				+O\left( n^{-1/2}e^{-t^2} \right)
		\end{equation}
		as $n\to\infty$, where $H_a(x)$ denotes the $a$-th Hermite polynomial.
\end{theorem}
\begin{proof}
	Set $\vec x=(x_0,\dots,x_{p-1})$ and $\vec y=(y_0,\dots,y_{p-1})$, and consider the more general quantity 
	\begin{align*}
		m_{n,h}^{(p)}(\vec x,\vec y) &= \det_{0\le i,j<p}\left( \binom{2n}{n+x_i-y_j}-\binom{2n}{n+h-x_i-y_j} \right).
	\end{align*}
	Factoring $\binom{2n}{n}$ out of each row of the determinant above and replacing each
	entry 
	with its asymptotic expansions as given in Lemma~\ref{lem:approx_binom_quot}, we find the asymptotics
	\[
	m_{n,h}^{(p)}(\vec x,\vec y) = \binom{2n}{n}^p\left( D_N(\vec x,\vec y)+O\left(e^{-h^2/n}n^{-1-2N}\right) \right),
	\qquad n\to\infty,
	\]
	where
	\[
	D_N(\vec x,\vec y)=
	\det_{0\le i,j<p}\left( 
	\sum_{u=0}^{4N+1}\left(
		\left(\frac{y_j-x_i}{\sqrt n}\right)^uT_{u;N}(0,n)
		-\left(\frac{y_j+x_i}{\sqrt n}\right)^uT_{u;N}(h,n)
		\right)\right)
	\]
	and $N>0$ is an arbitrary integer. Here, $T_{u;N}(h,n)$ is given by (see Lemma~\ref{lem:approx_binom_quot})
	\[
	T_{u;N}(h,n)=e^{-h^2/n}\left( 
		\frac{H_u(h/\sqrt n)}{u!}+\sum_{l=1}^{3N+1}n^{-l}\sum_{k=0}^{u-1}
					\sum_{r=1}^{2l}F_{r,l}\binom{2r}{u-k}\frac{(-1)^{u-k}}{k!}H_k\left( \frac{h}{\sqrt n}  \right)
					\left( \frac{h}{\sqrt n}  \right)^{2r+k-u}
	\right).
	\]
	
	The quantity $D_N(\vec x,\vec y)$ is seen to be polynomial in the $x_i$'s and $y_j$'s.
	This polynomial is divisible by the factors $(x_j-x_i)$ and $(y_j-y_i)$ for $0\le i<j<p$, for if $x_j=x_i$ then
	the $j$-th and the $i$-th row are equal and, therefore, the determinant is zero (if $y_j=y_i$ then the $j$-th and $i$-th column are equal).
	Hence,
	\begin{equation*}
	D_N(\vec x,\vec y)=
		n^{-\binom{p}{2}}
		\frac{\prod\limits_{0\le i<j<p}(x_j-x_i)(y_j-y_i)}{\prod\limits_{0\le j<p}j!^{2}}
		\chi(n,h)
		+O(n^{-1/2}e^{-h^2/n})
		,\qquad n\to\infty.
	\end{equation*}
	Here, the error term is determined by noting that every power of $x_j$ and $y_j$ entails a factor
	of $n^{-1/2}$, as can be seen from the definition of $D_N(\vec x,\vec y)$ above.
	The unknown coefficient $\chi(n,h)$ 
	can now be determined by comparing coefficients on both sides of the equation above.
	Comparing the coefficients of $\prod_{j=0}^{p-1}x_j^jy_j^j$, we obtain
	(after some simplifications) the equation
	\[
	\det_{0\le i,j<p}\left( (-1)^iH_{i+j}(0)-H_{i+j}\left( \frac{h}{\sqrt n} \right)e^{-h^2/n} \right)
	=\chi(n,h).
	\]
	If we specialise by setting $x_j=y_j=j$, then we see that 
	\[ m_{n,h}^{(p)}=n^{-\binom{p}{2}}\binom{2n}{n}^p
	\det_{0\le i,j<p}\left( (-1)^iH_{i+j}(0)-H_{i+j}\left( \frac{h}{\sqrt n} \right)e^{-h^2/n} \right)
	+O\left(n^{-1/2}e^{-h^2/n}  \right).
	\]
	Setting $h=t\sqrt n$ and replacing $m_{2n}^{(p)}$ with its asymptotic equivalent as given by
	Lemma~\ref{lem:total_num_asymptotics}, we obtain the result.
\end{proof}

\begin{figure}[width=8cm]
		\begin{center}
				\includegraphics{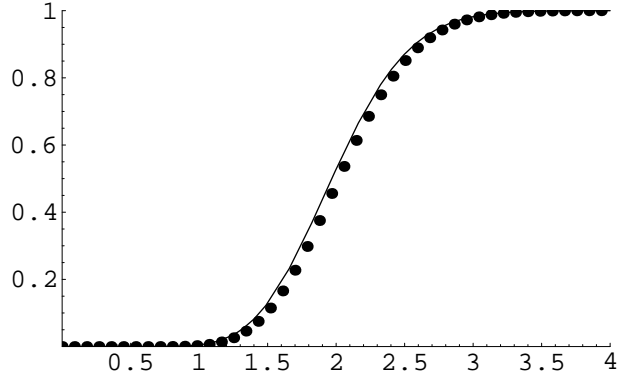}
		\end{center}
		\caption{Comparison of the c.d.f. of the random variable ``height'' on the set of $3-$watermelons of length $500$ 
		without wall (dotted curve) and the limiting distribution as given by Theorem~\ref{thm:height:weak_limit}.}
		\label{fig:cdf:height}
\end{figure}

Let us now turn our attention to the moments of the distribution of $H_{n,p}$. Clearly, we have for $s\in\N$,
\begin{equation}\label{eq:moments_exact}
	\E\left(H_{n,p}^{s}\right)=\sum_{h\ge 1}h^{s}\frac{m_{n,h+1}^{(p)}-m_{n,h}^{(p)}}{m_{n}^{(p)}}
=\sum_{h\ge 1}\left( h^s-(h-1)^s \right)\frac{m_{n}^{(p)}-m_{n,h}^{(p)}}{m_n^{(p)}}.
\end{equation}
The dominant terms of the asymptotics for the moments are going to be expressed by linear combinations of certain infinite exponential sums.
Asymptotics for these sums are to be determined now.

\begin{lemma}
		For $\nu\ge0$ and $\mu>0$ define
		\[ f_{\nu,\mu}(n)=\sum_{h\ge 1}h^\nu e^{-\mu h^2/n}. \]
		This sum admits the asymptotic series expansion
		\[
		f_{\nu,\mu}(n) \approx \frac{1}{2}\Gamma\left(\frac{ \nu+1}{2}\right)\left( \frac{n}{ \mu} \right)^{(\nu+1)/2}
		+\sum_{m\ge 0}\left( \frac{ \mu}{n} \right)^m \frac{(-1)^{\nu+m}B_{2m+\nu+1}}{(2m+\nu+1)!m!},
		\]
		as $n\to\infty$, where $\Gamma$ denotes the gamma function and $B_m$ is the $m$-th Bernoulli number defined via
		$\sum_{j\ge 0}B_jt^j/j!=t/(e^t-1)$.
		\label{lem:f_nu,mu}
\end{lemma}
\begin{proof}[Proof (Sketch)]
		Asymptotics for sums of this form can often be obtain by means of Mellin transform techniques. For a 
		detailed overview of Mellin transforms, harmonic sums and asymptotics, we refer to \cite{MR1337752}.

		We proceed with a sketch of the proof. The inverse Mellin transform gives
		\begin{align*}
				f_{\nu,\mu}(n)=\sum_{h\ge 1}h^\nu e^{-\mu h^2/n} =&
				\sum_{h\ge 1}\frac{h^\nu}{2\pi i}\int_{c-i\infty}^{c+i\infty}\Gamma(z)\left( \frac{\mu h^2}{n} \right)^{-z}dz \\
				=&\frac{1}{2\pi i}\int_{c-i\infty}^{c+i\infty}\Gamma(z)\left( \frac{\mu}{n} \right)^{-z}\zeta(2z-\nu)dz.
		\end{align*}
		The integrand has simple poles at $z=(\nu+1)/2$ and $z=0,-1,-2,\dots$ corresponding to the poles of the
		zeta and the gamma function, respectively.
		The result is now obtained by pushing the line of integration to the left and taking into account the residues.

		For the sake of convenience, we mention the evaluations
 		\begin{align*}
				\Res_{z=-m} \Gamma(z) &= \frac{(-1)^m}{m!}, \qquad m=0,1,2,\dots, \\
				\Res_{z=1} \zeta(z) &= 1 \\
				\zeta(-m) &= B_{m+1}\frac{(-1)^m}{m+1}, \qquad m=0,1,2,\dots, 
		\end{align*}
		where $B_m$ denotes the $m$-th Bernoulli number defined via  $\sum_{j\ge 0}B_j t^j/j!=t/(e^t-1)$.
\end{proof}

The rest of this section is devoted to the proof of Theorem~\ref{thm:height_moments} below, which gives the final 
expression for the asymptotics of the moments. In order to present the proof of this theorem in a clear fashion
we split it into a series of lemmas. For a more detailed overview of the proof, we refer directly to the proof of
Theorem~\ref{thm:height_moments}.

As a first step, we prove in Lemma~\ref{lem:height_moments} a preliminary asymptotic expression for the moments of
the height distribution. The presented compact form of the asymptotics makes use of certain linear operators
that are going to be defined now.
\begin{definition}\label{def:Xi10}
	Let $\Xi_1$ and $\Xi_0$ denote the linear operators 
	defined by
	\begin{align*}
			\Xi_1\left(h^{\nu}e^{-\mu h^2}\right) &= 
				\frac{1}{2}\Gamma\left(\frac{ \nu+1}{2}\right)\left( \frac{1}{ \mu} \right)^{(\nu+1)/2} \\
			\Xi_0\left(h^{\nu}e^{-\mu h^2}\right) &= 
			(-1)^\nu\frac{B_{\nu+1}}{(\nu+1)!},
	\end{align*}
	where $B_k$ denotes the $k$-th Bernoulli number.
\end{definition}
By Lemma~\ref{lem:f_nu,mu} we have
\[
	f_{\nu,\mu}(n) = \Xi_1\left( h^{\nu}e^{-\mu h^2} \right)n^{(\nu+1)/2}+\Xi_0\left(  h^{\nu}e^{-\mu h^2} \right)
	+O(n^{-1}),\qquad n\to\infty,
\]
so that $\Xi_1$ and $\Xi_0$ yield the coefficients of the first two terms in the asymptotic expansion of $f_{\nu,\mu}(n)$.

The preliminary expression for the asymptotics of the moments can now be proven in pretty much the same way as in
Theorem~\ref{thm:height:weak_limit}.
\begin{lemma}\label{lem:height_moments}
	For $s\in\N$, $s\ge 1$, the $s$-th moment of the random variable ``height'' satisfies the
	asymptotics
	\begin{equation}\label{eq:height_moments}
			\E\left( H_{n,p}^s\right)=
			s\Xi_1\left(\kappa_{p}h^{s-1}\right)n^{s/2}
			-\Xi_1\left(\binom{s}{2} \kappa_{p}h^{s-2}+\tau_{p}h^{s-1}\right)n^{(s-1)/2}
			+\Xi_0(\kappa_{p})
			+O\left( n^{s/2-1} \right)
	\end{equation}
	as $n\to\infty$, where
	\[
		\kappa_p=1-\frac{2^{-\binom{p}{2}}}{\prod\limits_{0\le j<p}j!}
		\det_{0\le i,j<p}\left( (-1)^iH_{i+j}(0)-H_{i+j}\left( h \right)e^{-h^2} \right)
	\]
	and
	\[
	\tau_p=(p-1)\frac{2^{-\binom{p}{2}}}{\prod\limits_{0\le j<p}j!}
				\det_{0\le i,j<p}\left(
			\begin{cases}
				(-1)^iH_{i+j}(0)-H_{i+j}\left(h\right)e^{-h^2} & \textrm{if $i<p-1$} \\
				(-1)^pH_{p+j}(0)-H_{p+j}\left(h\right)e^{-h^2} & \textrm{if $i=p-1$}
			\end{cases}
		\right).
	\]
Here, $H_k(z)$ denotes the $k$-th Hermite polynomial.
\end{lemma}
\begin{proof}
	Recall the exact expression for the $s$-th moment of the random variable ``height'' (see Equation~\eqref{eq:moments_exact}),
	\begin{equation}\label{eq:height_moments_exact_proof}
	\E\left( H_{n,p}^s \right)=\sum_{h=1}^{n+2p-2}\left( h^s-(h-1)^s \right)\frac{m_n^{(p)}-m_{n,h}^{(p)}}{m_n^{(p)}}.
	\end{equation}
	Asymptotics for this quantity can be obtained in pretty much the same way as Theorem~\ref{thm:height:weak_limit}.
	Compared to the problem of determining asymptotics for \eqref{eq:height_cdf}, the main difference 
	now is the summation over $h$.

	We consider the more general quantity
	\begin{align*}
		m_n^{(p)}(\vec x,\vec y)-m_{n,h}^{(p)}(\vec x,\vec y) =
		\det_{0\le i,j<p}\left( \binom{2n}{n+x_i-y_j}\right)
		-\det_{0\le i,j<p}\left( \binom{2n}{n+x_i-y_j}-\binom{2n}{n+h-x_i-y_j} \right),
	\end{align*}
	where $\vec x=(x_0,\dots,x_{p-1})$ and $\vec y=(y_0,\dots,y_{p-1})$.
	As a first step, we pull $\binom{2n}{n}$ out of each row of the determinants above.
	Now, we restrict the range of summation in \eqref{eq:height_moments_exact_proof}
	to $1\le h\le n^{1/2+\varepsilon}$ for some $\varepsilon>0$. This truncation is justified by
	Stirling's formula, which shows that
	\[
	\frac{\binom{2n}{n+\alpha}}{\binom{2n}{n}}=O\left( e^{-n^{2\varepsilon}} \right),\qquad n\to\infty,
	\]
	whenever $|\alpha|\ge n^{1/2+\varepsilon}$.
	This implies that the total contribution of all summands in \eqref{eq:height_moments_exact_proof} satisfying
	$h>n^{1/2+\varepsilon}$ is exponentially small as $n\to\infty$ and, therefore, negligible.
	In all the remaining summands we replace all the quotients of binomial coefficients with their
	asymptotic expansions as given in Lemma~\ref{lem:approx_binom_quot}.
	Finally, we re-extend the range of summation to $h\ge 1$, which, again, introduces an exponentially small
	error term.
	This gives the asymptotics
	\[
		\E\left( H_{n,p}^s \right)
		=\sum_{h\ge 1}\left(\left( h^s-(h-1)^s \right)
		\left(\frac{\binom{2n}{n}^p}{m_n^{(p)}}D_N(\vec e,\vec e)
		+O\left(e^{-h^2/n}n^{\binom{p}{2}-1-2N}\right)\right)
		\right), \qquad n\to\infty,
	\]
	where $\vec e=(0,1,\dots,p-1)$.
	Here, the structure of the error term is a consequence of Lemma~\ref{lem:total_num_asymptotics}, and
	the quantity $D_N(\vec x,\vec y)$ is defined by
	\begin{multline}\label{eq:D_N_shape}
	D_N(\vec x,\vec y)=
	\det_{0\le i,j<p}\left( 
	\sum_{u=0}^{4N+1}\left(
		\left(\frac{y_j-x_i}{\sqrt n}\right)^uT_{u;N}(0,n)\right)\right) \\
	-\det_{0\le i,j<p}\left( 
	\sum_{u=0}^{4N+1}\left(
		\left(\frac{y_j-x_i}{\sqrt n}\right)^uT_{u;N}(0,n)
		-\left(\frac{y_j+x_i}{\sqrt n}\right)^uT_{u;N}(h,n)
		\right)\right),
	\end{multline}
	where $N>0$ is an arbitrary integer and
	\[
	T_{u;N}(h,n)=e^{-h^2/n}\left( 
		\frac{H_u(h/\sqrt n)}{u!}+\sum_{l=1}^{3N+1}n^{-l}\sum_{k=0}^{u-1}
		\sum_{r=1}^{2l}F_{r,l}\binom{2r}{u-k}\frac{H_k\left( h/\sqrt n  \right)}{k!}
					\left(- \frac{h}{\sqrt n}  \right)^{2r+k-u}
	\right).
	\]
	
	As a consequence of Lemma~\ref{lem:f_nu,mu}, we see (after expanding the term $(h-1)^s$) that
	\[
	\sum_{h\ge 1}\left(\left( h^s-(h-1)^s \right)O\left(e^{-h^2/n}n^{\binom{p}{2}-1-2N}\right)\right)
	= O\left( n^{\binom{p}{2}-2N+(s-1)/2} \right),
	\]
	which is negligible for sufficiently large $N$.	Hence, we have the asymptotics
	\[
		\E\left( H_{n,p}^s \right)
		=\frac{\binom{2n}{n}^p}{m_n^{(p)}}\sum_{h\ge 1}\Big(\left( h^s-(h-1)^s \right)
		D_N(\vec e,\vec e)	\Big)
		+O\left(n^{\binom{p}{2}-2N+(s-1)/2}\right)	, \qquad n\to\infty.
	\]
	It remains to determine the part of $D_N(\vec x,\vec y)$ that gives the dominant contribution to
	the asymptotics above.
	First, we note that $D_N(\vec x, \vec y)$ is a polynomial in the $x_i$'s and $y_i$'s.
	Obviously, $D_N(\vec x,\vec y)$ is equal to zero whenever $x_i=x_j$ or $y_i=y_j$ for some $i\neq j$, for if
	$x_i=x_j$ ($y_i=y_j$) then the $i$-th and $j$-th rows (columns) of the determinants involved in the definition
	of $D_N(\vec x,\vec y)$ are equal, and, therefore, the determinants are equal to zero.
	This implies that $D_N(\vec x,\vec y)$ is of the form
	\[
	D_N(\vec x,\vec y)=
		n^{-\binom{p}{2}}
		\frac{\prod\limits_{0\le i<j<p}(x_j-x_i)(y_j-y_i)}{\prod\limits_{0\le j<p}j!^{2}}
		\left( \chi(n,h)+\sum_{j=0}^{p-1}\left( \xi_j(n,h)\frac{x_j}{\sqrt n}+\eta_j(n,h)\frac{y_j}{\sqrt n} \right)
		+O\left( n^{-1}e^{-h^2/n} \right) \right)
	\]
	as $n\to\infty$.
	By comparing coefficients of $\prod_{j=0}^{p-1}x_j^jy_j^j$ on both sides of the equation above, 
	we have already seen (see Theorem~\ref{thm:height:weak_limit}) that
	\[
	\chi(n,h)=
	\det_{0\le i,j<p}\left( (-1)^{i}H_{i+j}(0)\right)-
	\det_{0\le i,j<p}\left( (-1)^iH_{i+j}(0)-H_{i+j}\left( \frac{h}{\sqrt n} \right)e^{-h^2/n} \right)
	.
	\]
	Analogously we can determine $\xi_k(n,h)$.
	By comparing the coefficients of $x_k\prod_{j=0}^{p-1}x_j^jy_j^j$ on both sides of the
	equation above we obtain the equations
	\[ 0=\xi_{k}(n,h)-\xi_{k+1}(n,h),\qquad k<p-1, \]
	and
	\[
	\xi_{p-1}(n,h) =
		-\frac{1}{p}
		\det_{0\le i,j<p}\left(
			\begin{cases}
				(-1)^iH_{i+j}(0)-H_{i+j}\left( \frac{h}{\sqrt n} \right)e^{-h^2/n} & \textrm{if $i<p-1$} \\
				(-1)^pH_{p+j}(0)-H_{p+j}\left( \frac{h}{\sqrt n} \right)e^{-h^2/n} & \textrm{if $i=p-1$}
			\end{cases}
		\right).
	\]
	Note, that the coefficient of $x_k\prod_{j=0}^{p-1}x_j^jy_j^j$ in the first determinant of \eqref{eq:D_N_shape} is
	equal to zero, which is easily seen to be true for $k<p-1$, and for $k=p-1$ this is seen to be true
	by a series of column and row operations that yield a new matrix consisting of two non-square blocks.
	Similar expressions (with $i$ and $j$ interchanged) can be found for the $\eta_k(n,h)$, $0\le k<p$.
	
	Noting that $H_{i+j}(0)$ is non-zero if and only if $i+j$ is even we deduce that
	$(-1)^{i}H_{i+j}(0)=(-1)^jH_{i+j}(0)$, which implies
	\[ \xi_{p-1}(n,h)=\eta_{p-1}(n,h), \]
	and also
	\[
	\det_{0\le i,j<p}\left( (-1)^{i}H_{i+j}(0)\right) = \det_{0\le i,j<p}\left( (-1)^{(i+j)/2}H_{i+j}(0)\right)
	=2^{\binom{p}{2}}\prod_{j=0}^{p-1}j!.
	\]
	Here, the last equality has been proven in Lemma~\ref{lem:detEval}.

	If we specialise to $x_j=y_j=j$, $0\le j<p$, then we obtain
	\[
		D_N(\vec e,\vec e)=
		n^{-\binom{p}{2}}
		\left( \chi(n,h)+2\binom{p}{2}\xi_{p-1}(n,h)n^{-1/2} \right)+O\left(n^{-1} e^{-h^2/n} \right),
		\qquad n\to\infty,
	\]
	where $\vec e = (0,1,\dots,p-1)$.

	Choosing $N$ large enough and expanding the term $h^s-(h-1)^s$ in the asymptotics for 
	$\E\left( H_{n,p}^s \right)$ above, we obtain with the help of Lemma~\ref{lem:f_nu,mu}
	the asymptotics
	\[
	\E\left( H_{n,p}^s \right)=
	\frac{\binom{2n}{n}^p}{m_n^{(p)}}
	\sum_{h\ge 1}\left( sh^{s-1}-\binom{s}{2}h^{s-2} \right)D_N(\vec e,\vec e)
	+O\left( n^{s/2-1} \right),
	\qquad n\to\infty,
	\]
	and replacing $D_N(\vec e,\vec e)$ with its asymptotic expansion as given above proves the lemma.
\end{proof}

\begin{lemma}\label{lem:detEval}
Let $H_k(x)$ denote the $k$-th Hermite polynomial as defined by Equation~\eqref{eq:hermitepolynomial}.
We have the determinant evaluation
\begin{equation}
\det_{0\le i,j<p}\left( (-1)^{(i+j)/2}H_{i+j}(0) \right)=
2^{\binom{p}{2}}\prod_{j=0}^{p-1}j!.
\end{equation}
\end{lemma}
\begin{proof}	
	The determinant under consideration is a Hankel determinant. Therefore, we can hope to evaluate it
	with the help of orthogonal polynomials (for details see~\cite[Section 2.7]{MR1701596}).
	It is well known (see, e.g., \cite[page 105]{MR0106295}) that for $k\in\N$ we have
	\[ H_{2k+1}(0)=0 \qquad\textrm{and}\qquad H_{2k}(0)=(-1)^k\frac{(2k)!}{k!}. \]
	Consequently, we obtain
	\[
	\det_{0\le i,j<p}\left( (-1)^{(i+j)/2}H_{i+j}(0) \right)=
	2^{\binom{p}{2}}\det_{0\le i,j<p}\left( \frac{1+(-1)^{i+j}}{2}\frac{2^{(i+j)/2}}{\sqrt\pi}\Gamma\left(\frac{i+j+1}{2} \right) \right).
	\]

	The $(i,j)$-th entry of the determinant on the right hand side above is seen to be precisely the $(i+j)$-th moment with
	respect to the Gaussian weight $w(x)=\frac{1}{\sqrt{2\pi}}e^{-x^2/2}$ on $\R$, that is,
	\[
	\frac{1}{\sqrt{2\pi}}\int_{-\infty}^{\infty}x^ke^{-x^2/2}dx = 
	\frac{1+(-1)^{k}}{2}\frac{2^{k/2}}{\sqrt\pi}\Gamma\left(\frac{k+1}{2} \right),\qquad k=0,1,2\dots
	\]

	The family of monic orthogonal polynomials associated with the weight $w(x)$ is given by
	\begin{equation}\label{eq:orth_poly_h}
			2^{-n/2}H_n\left( \frac{x}{\sqrt{2}} \right),\qquad n=0,1,2,\dots
	\end{equation}
	where $H_n(x)$ denotes the $n$-th Hermite polynomial as defined by Equation~\eqref{eq:hermitepolynomial}.
	The three term recursion relation for the orthogonal polynomials~(\ref{eq:orth_poly_h}) is seen to be (cf. \cite[p.105]{MR0106295})
	\[ 2^{-(n+1)/2}H_{n+1}\left( \frac{x}{\sqrt 2} \right)
	=x2^{-n/2}H_{n}\left( \frac{x}{\sqrt 2} \right)-n2^{-(n-1)/2}H_{n-1}\left( \frac{x}{\sqrt 2} \right),
	\qquad n=1,2,\dots, \]
	with the initial values $H_0\left(\frac{x}{\sqrt 2} \right)=1$ and $2^{-1/2}H_1\left(\frac{x}{\sqrt 2} \right)=x$.
	Now, an application of \cite[Theorem 11]{MR1701596}) shows that
%
	\[
	\det_{0\le i,j<p}\left( \frac{1+(-1)^{i+j}}{2}\frac{2^{(i+j)/2}}{\sqrt\pi}\Gamma\left(\frac{i+j+1}{2} \right) \right)
	=\prod_{j=0}^{p-1}j!,
	\]
	which proves the claim.
\end{proof}
\begin{lemma}\label{lem:Xi1_property}
Let $\mu>0$ denote a real number.
The operator $\Xi_1$ from Definition~\ref{def:Xi10} satisfies the relation
\begin{equation}\label{eq:Xi1_relation}
\Xi_1\left( \frac{d}{dh}\left( h^\nu e^{-\mu h^2} \right) \right)=
\begin{cases}
	-1 & \textrm{if $\nu=0$} \\
	0 & \textrm{if $\nu>0$}.
\end{cases}
\end{equation}
\end{lemma}
\begin{proof}
	For $\nu=0$ the claim follows immediately from the definition of the operator $\Xi_1$.
	For $\nu>0$ we calculate
	\[
		\Xi_1\left( h^{\nu+1}e^{-\mu h^2} \right)=\frac{\nu}{2\mu}\Xi_1\left( h^{\nu-1}e^{-\mu h^2} \right),
	\]
	from which the claims follows upon multiplying by $2\mu$ and rearranging the terms.
\end{proof}

The next result is not obvious at all, and, on the contrary, is a quite surprising fact.
\begin{lemma}\label{lem:kappa_tau_relation}
Let $\kappa_p$ and $\tau_p$ denote the determinants defined in Lemma~\ref{lem:height_moments}.
We have the relation
\begin{equation}\label{eq:kappa_tau_relation}
(p-1)\frac{d}{dh}\kappa_p=\tau_p,\qquad p\ge 1.
\end{equation}
\end{lemma}
\begin{proof}
For the sake of convenience we set
\[ C=2^{-\binom{p}{2}}\left( \prod_{j=0}^{p-1}j! \right)^{-1}. \]
The derivative of a $p\times p$ determinant is the sum $p$ determinants, where the $j$-th addend is equal to 
the original determinant with the $j$-th row replaced by its derivative. Hence,
\[
\frac{d}{dh}\kappa_p = C\left( \sum_{j=0}^{p-2}M_j \right)+CM_{p-1},
\]
where
\[
M_i=
\det\left(\ 
\begin{pmatrix}
	\H_{0,0} & \cdots & \H_{0,p-1} \\ \vdots & \ddots & \vdots \\ \H_{i-1,0} & \cdots & \H_{i-1,p-1} \\
	-H_{i+1}(h)e^{-h^2} & \cdots & -H_{i+p}(h)e^{-h^2} \\
	\H_{i+1,0} & \cdots & \H_{i+1,p-1} \\ \vdots & \ddots & \vdots \\ \H_{p-1,0} & \cdots & \H_{p-1,p-1}
\end{pmatrix}
\ \right)
,
\qquad
\H_{i,j}=(-1)^iH_{i+j}(0)-H_{i+j}(h)e^{-h^2}.
\]
%
We want to mention that $(p-1)CM_{p-1}$ is equal to the expression for $\tau_p$ except for the constant terms
in the last row of the determinant.

For $0\le i<p-1$ the quantity $M_i$ can also be represented by the expression
\[
M_i=
\det\left(\ 
\begin{pmatrix}
	\H_{0,0} & \cdots & \H_{0,p-1} \\ \vdots & \ddots & \vdots \\ \H_{i-1,0} & \cdots & \H_{i-1,p-1} \\
	\H_{i+1}& \cdots & \H_{i+p} \\
	(-1)^{i+1}H_{i+1}(0) & \cdots & (-1)^{i+1}H_{i+p}(0) \\
	\H_{i+2,0} & \cdots & \H_{i+2,p-1} \\ \vdots & \ddots & \vdots \\ \H_{p-1,0} & \cdots & \H_{p-1,p-1}
\end{pmatrix}
\ \right),
\qquad 0\le i<p-1,
\]
which is more convenient to work with.

The Laplace expansion for determinants with respect to the row $j+1$, $0\le j<p-1$, gives
\[ M_j = \sum_{k=0}^{p-1}(-1)^{j+1}H_{j+1+k}(0)M_{j,k}, \qquad 0\le j<p-1, \]
where $M_{j,k}$ denotes the minor of $M_j$ obtained by removing row $j+1$ and column $k$, i.e.,
\[
M_{j,k}= \det
\left(\ 
\begin{pmatrix}
\H_{0,0} & \cdots & \H_{0,k-1} & \H_{0,k+1} & \cdots & \H_{0,p-1} \\
\vdots & \ddots & \vdots & \vdots & \ddots & \vdots \\
\H_{i-1,0} & \cdots & \H_{i-1,k-1} & \H_{i-1,k+1} & \cdots & \H_{i-1,p-1} \\
\H_{i+1,0} & \cdots & \H_{i+1,k-1} & \H_{i+1,k+1} & \cdots & \H_{i+1,p-1} \\
\H_{i+2,0} & \cdots & \H_{i+2,k-1} & \H_{i+2,k+1} & \cdots & \H_{i+2,p-1} \\
\vdots & \ddots & \vdots & \vdots & \ddots & \vdots \\
\H_{p-1,0} & \cdots & \H_{p-1,k-1} & \H_{p-1,k+1} & \cdots & \H_{p-1,p-1}
\end{pmatrix}
\ \right).
\]

Now, consider the sum
\[
\sum_{j=0}^{p-2}M_j = 
\left(\sum_{j=0}^{p-2}\sum_{k=0}^{p-2}(-1)^{j+1}H_{j+1+k}(0)M_{j,k}\right)
+\sum_{j=0}^{p-2}(-1)^{j+1}H_{j+p}(0)M_{j,p-1}.
\]
The first sum on the right hand side in fact is equal to zero as is going to be shown now.
First, note that
\[ M_{j,k} = M_{k,j} \]
since the matrices involved are transposes of each other.
Recalling that $H_k(0)$ is non zero if and only if $k$ is an even number we deduce that
\[
(-1)^{j+1}H_{j+1+k}(0)M_{j,k}=-(-1)^{k+1}H_{k+1+j}(0)M_{k,j},
\]
and both expressions correspond to different addends of the double sum above ($j+1+k$ has to be even). 
This shows that the value of the double sum is indeed equal to zero.

For the second sum we have
\begin{align*}
\sum_{j=0}^{p-2}(-1)^{j+1}H_{j+p}(0)M_{j,p-1}
&= -\sum_{j=0}^{p-2}(-1)^{p}H_{j+p}(0)M_{p-1,j} \\
&= \det_{0\le k,l<p}\left( 
\begin{cases}
(-1)^kH_{k+l}(0)-H_{k+l}(h)e^{-h^2} & \textrm{if $k < p-1$} \\
(-1)^pH_{p+l} & \textrm{if $k=p-1$}
\end{cases}
\right),
\end{align*}
which proves the lemma.
\end{proof}

We are now able to to state and prove the final expression for the asymptotics of the moments.
\begin{theorem}\label{thm:height_moments}
	The expected value of the random variable $H_{n,p}$ satisfies the asymptotics
	\begin{equation}\label{eq:height_moments_expectedvalue}
		\E\left( H_{n,p}\right) =  \Xi_1\left( \kappa_p \right)\sqrt{n}+p-\frac{3}{2}+O\left( n^{-1/2} \right),
		\qquad n\to\infty,		
	\end{equation}
	and for $s\in\N$, $s\ge 2$, we have the asymptotics
	\begin{equation}\label{eq:height_moments_higher}
		\E\left( H_{n,p}^s\right) = s\Xi_1(\kappa_ph^{s-1})n^{s/2}
		+(s-1)\left( p-1-\frac{s}{2} \right)\Xi_1\left( \kappa_ph^{s-2} \right)n^{(s-1)/2}
		+O\left( n^{s/2-1} \right),
		\qquad n\to\infty.		
	\end{equation}
	Here, $\kappa_p$ is defined by
	\[
	\kappa_p=1-\frac{2^{-\binom{p}{2}}}{\prod\limits_{0\le j<p}j!}
	\det_{0\le i,j<p}\left( (-1)^iH_{i+j}(0)-H_{i+j}\left( h \right)e^{-h^2} \right),
	\]
	where $H_k(z)$ denotes the $k$-th Hermite polynomial.
\end{theorem}
\begin{proof}
	As a first step we need to establish some simple facts concerning the quantity $\kappa_p$.
	To be more precise, we have to show that $\kappa_p$ is an even function with respect to $h$ that has
	no constant term, i.e., is of the form
	\[ \kappa_p = \sum_{k=0}^{K}\sum_{m=1}^{M}\lambda_{k,m}h^{2k}e^{-mh^2} \]
	for some numbers $K$, $M$ and some constants $\lambda_{k,m}$.

	It is obvious from the definition of the Hermite polynomials (see Equation~\eqref{eq:hermitepolynomial})
	that the $k$-th Hermite polynomial is an even (odd) polynomial whenever $k$ is even (odd).
	This also implies the equality $(-1)^iH_{i+j}(0)=(-1)^jH_{i+j}(0)$.
	Now, replacing $h$ by $-h$ in the definition of $\kappa_p$, factoring $(-1)^i$ out of the $i$-th row and
	$(-1)^j$ out of the $j$-th row we see that the expression remains unaltered. Hence, $\kappa_p$ is an even function of $h$.
	The constant term of $\kappa_p$ is seen to be equal to
	\[
	1-\frac{2^{-\binom{p}{2}}}{\prod_{j=0}^{p-1}j!}\det_{0\le i,j<p}\left( (-1)^iH_{i+j}(0) \right)
	=1-\frac{2^{-\binom{p}{2}}}{\prod_{j=0}^{p-1}j!}\det_{0\le i,j<p}\left( (-1)^{(i+j)/2}H_{i+j}(0) \right)=0,
	\]
	where the last equality is a consequence of Lemma~\ref{lem:detEval}.
	This proves the claimed form of $\kappa_p$.

	We are now going to prove the asymptotics~(\ref{eq:height_moments_higher}). Therefore, we assume that $s>1$.
	The properties of $\kappa_p$ established above together with Lemma~\ref{lem:Xi1_property} imply the equation
	\[ \Xi_1\left( \frac{d}{dh}\left( \kappa_ph^{s-1} \right) \right)=0, \]
	and the product rule for the derivative together with Lemma~\ref{lem:kappa_tau_relation} show that
	\[ \Xi_1\left( \tau_ph^{s-1} \right)=-(s-1)(p-1)\Xi_1\left( \kappa_ph^{s-2} \right). \]
	The asymptotics~(\ref{eq:height_moments_higher}) is now obtained from the asymptotics~(\ref{eq:height_moments})
	upon noting that the $\Xi_0$-term is negligible for $s\ge2$.

	Finally, we prove the asymptotics~(\ref{eq:height_moments_expectedvalue}) and, therefore, assume $s=1$.
	For the sake of simplicity we set
	\[ C = 2^{-\binom{p}{2}}\left( \prod_{j=0}^{p-1}j! \right)^{-1}. \]
	From Lemma~\ref{lem:kappa_tau_relation} and Lemma~\ref{lem:Xi1_property} we deduce that
	\[
	\Xi_1\left( \tau_p \right)=(p-1)\Xi_1\left( \frac{d}{dh}\kappa_p \right)
	=-(p-1)\Xi_1\left( C\frac{d}{dh}\chi(h) \right),
	\]
	where
	\[ \chi(h)=\det_{0\le i,j<p}\left( (-1)^iH_{i+j}(0)-H_{i+j}(0)e^{-h^2} \right). \]
	This last determinant can be evaluated to a closed form expression with the help of Lemma~\ref{lem:detEval}.
	Factoring $1-(-1)^je^{-h^2}$ out of each column of the determinant we see that
	\begin{align*}
		\chi(h)
		&=\left( \prod_{j=0}^{p-1}\left( 1-(-1)^je^{-h^2} \right) \right)\det_{0\le i,j<p}\left( (-1)^{(i+j)/2}H_{i+j}(0) \right) \\
		&=\frac{1}{C}\left(1-e^{-2h^2}\right)^{\floor{p/2}}\left(1-e^{-h^2}\right)^{\ceil{p/2} -\floor{p/2}}.
	\end{align*}
	Now, an application of Lemma~\ref{lem:Xi1_property} shows that
	\[ \Xi_1\left( \frac{d}{dh}\chi(h) \right)=-1, \]
	which implies
	\[ \Xi_1\left( \tau_p \right)=1-p. \]

	The last step of the proof is the evaluation of the quantity $\Xi_0(\kappa_p)$. Recalling that $\kappa_p$ is
	an even function with respect to $h$ as well as the fact that all odd Bernoulli numbers except for $B_1$ are zero,
	i.e., $B_{2\nu+1}=0$, $\nu\ge 1$, we deduce the equation
	\[
	\Xi_0(\kappa_p)
	=\Xi_0\left( 1-C\chi(h) \right)
	=\Xi_0\left( 1-\left(1-e^{-2h^2}\right)^{\floor{p/2}}\left(1-e^{-h^2}\right)^{\ceil{p/2} -\floor{p/2}} \right).
	\]
	The definition of $\Xi_0$ reveals that $\Xi_0\left( h^\nu e^{-\mu h^2} \right)$ is independent of $\mu$.
	Consequently, we see that
	\[ \Xi_0(\kappa_p) = B_1=-\frac{1}{2}. \]
	This proves the asymptotics~(\ref{eq:height_moments_expectedvalue}) and completes the proof of the theorem.
\end{proof}
\begin{table}
\caption{
This table gives the coefficient of the dominant asymptotic term of $\E H_{n,p}^s$ as $n\to\infty$ for small values of $s$ and $p$
(see  Theorem~\ref{thm:height_moments}).}
\begin{tabular}{c|r|r|r}
\label{tab:moment_height_dominant_coef}
		$s\kappa_{s}^{(p)}$ & \centering{$s=1$} & $s=2$ & $s=3$\\
		\hline
		$p=1$ & $\frac{1}{2}\sqrt\pi=0.88\dots$ & $1$ & $\frac{3}{4}\sqrt\pi=1.32\dots$\\
		$p=2$ & $\frac{2+\sqrt 2}{4}\sqrt\pi=1.51\dots$ & $\frac{5}{2}$ & $\frac{3(12+\sqrt 2)}{16}\sqrt\pi=4.45\dots$\\
		$p=3$ & $\frac{72+45\sqrt 2-16\sqrt 3}{96}\sqrt\pi=1.99\dots$ & $\frac{25}{6}$ & $\frac{1584+315\sqrt 2-32\sqrt 3}{385}\sqrt\pi=9.11\dots$\\
		$p=4$ & $\frac{10368+17091\sqrt 2-3776\sqrt 3}{20736}\sqrt{\pi}=2.39\dots$ & $\frac{1915}{324}$ & $\frac{520992+165969\sqrt 2-29824\sqrt{3}}{82944}\sqrt\pi=15.04\dots$\\
\end{tabular}
\end{table}
Table~\ref{tab:moment_height_dominant_coef} shows the constant of the dominant asymptotic term as $n\to\infty$ for the
$s$-th moment of the height distribution for small values of $s$ and $p$.

\section{Range}\label{sec:range}
We determine the asymptotics for $n\to\infty$ of
\begin{equation}\label{eq:range_cdf_exact}
\P\left\{ R_{n,p}\le r \right\}
	=\frac{1}{m_{n}^{(p)}}\sum_{h=2p-2}^{r}\left( m_{n,h+1,r-h+1}^{(p)}-m_{n,h,r-h+1}^{(p)} \right).
\end{equation}
Note that $m_{n,h+1,r-h+1}^{(p)}-m_{n,h,r-h+1}^{(p)}$ is the number of watermelons with height exactly $h$ and range $\le r$.

\begin{theorem}\label{thm:range_weaklimit}
	For each fixed $t\in(0,\infty)$ we have the asymptotics
	\begin{equation}
			\P\left\{ \frac{R_{n,p}+1}{\sqrt n}\le t \right\}\to
			\frac{2^{-\binom{p}{2}}}{\prod_{i=0}^{p-1}i!}
			\int_0^t\left(\left. \frac{d}{dz}T_p(z,w)\right|_{z=t} \right)dw,
			\qquad n\to\infty,
	\end{equation}
	where
\[ T_p(z,w)=\det_{0\le i,j<p}\left( (-1)^i\left(\sum_{\ell\in\Z}H_{i+j}(\ell z)e^{-(\ell z)^2} \right)
-\left(\sum_{\ell\in\Z}H_{i+j}\left( \ell z+w \right)e^{-(\ell z+w)^2}\right)\right).
\]
%
Here, $H_a$ denotes the $a$-th Hermite polynomial.
\end{theorem}
\begin{proof}
Since $m_{n,2p-2,k}^{(p)}=0$ for any $k$, Equation~\eqref{eq:range_cdf_exact} can be rewritten as
\[
\P\left\{ R_{n,p}\le r \right\}
	=\frac{m_{n,r+1,1}^{(p)}}{m_{n}^{(p)}}
	+\frac{1}{m_{n}^{(p)}}\sum_{h=2p-1}^{r}\left( m_{n,h,r-h+2}^{(p)}-m_{n,h,r-h+1}^{(p)}  \right).
\]
The first term on the right-hand side is negligible. To see this, we note that $m_{n,r+1,1}$ is equal to the number of $p$-watermelons
with wall and height $\le r$, which is of order $\binom{2n}{n}^pn^{-p^2}$ as $n\to\infty$ (see \cite{watermelons:withwall} for details),
whereas $m_n^{(p)}$ is of order $\binom{2n}{n}^pn^{-\binom{p}{2}}$ (see Lemma~\ref{lem:total_num_asymptotics}).

Asymptotics for the sum on the right-hand side can now be established in a fashion analogous to the proof
of Theorem~\ref{thm:height:weak_limit}.  A more detailed presentation of these techniques can also be found
in \cite[Theorem 2]{watermelons:withwall}.
We find the asymptotics
\[
	\P\left\{ R_{n,p}\le r \right\} \sim
	\frac{\binom{2n}{n}^pn^{-\binom{p}{2}}}{m_{n}^{(p)}} 
	\sum_{h=2p-1}^{r}\left(T_p\left( \frac{r+2}{\sqrt n},\frac{h}{\sqrt n} \right)- T_p\left( \frac{r+1}{\sqrt n},\frac{h}{\sqrt n} \right) \right)
\]
as $n\to\infty$, where
\[
T_p(t,w)=\det_{0\le i,j<p}\left(
			(-1)^i\left(\sum_{\ell\in\Z}H_{i+j}(\ell t)e^{-(\ell t)^2} \right)
			-\left(\sum_{\ell\in\Z}H_{i+j}\left( \ell t+w \right)e^{-(\ell t+w)^2}\right)
\right).
\]

Now, Taylor series expansion shows that
\[
T_p\left( \frac{r+2}{\sqrt n},\frac{h}{\sqrt n} \right)- T_p\left( \frac{r+1}{\sqrt n},\frac{h}{\sqrt n} \right)
=\frac{1}{\sqrt n}T_p'\left( \frac{r+1}{\sqrt n},\frac{h}{\sqrt n} \right)+O\left( n^{-1} \right),\qquad n\to\infty,
\]
where $T'$ denotes the derivative of $T$ with respect to its first argument.
Setting $r+1=t\sqrt n$ we see that
\[
\sum_{h=2p-1}^{r}\left(T_p\left( \frac{r+2}{\sqrt n},\frac{h}{\sqrt n} \right)- T_p\left( \frac{r+1}{\sqrt n},\frac{h}{\sqrt n} \right) \right)
\sim \sum_{h=2p-1}^{r}\frac{1}{\sqrt n}T_p'\left( \frac{r+1}{\sqrt n},\frac{h}{\sqrt n} \right)
\to \int_0^{t}T'\left( t,w \right)dw
\]
as $n\to\infty$.

\end{proof}
\begin{remark}
	For the special case $p=1$ we recover a well-known fact originally proven by 
	Chung~\cite{MR0467948} and Kennedy~\cite{MR0402955}. Namely, the equality of the distributions
	of the height of Brownian excursions and the range of Brownian bridges. This result also
	follows from a more general relation between excursions an bridges proved by Vervaat~\cite{0392.60058}.

	In fact, for $p=1$ we have
	\[ \left.\frac{d}{dz}T_1(z,w)\right|_{z=t}=-\sum_{\ell\in\Z}2\ell^2te^{-(\ell t)^2}+2\sum_{\ell\in\Z}\ell(\ell t+w)e^{-(\ell t+w)^2}, \]
	which shows that
	\[ 	
	\P\left\{ \frac{R_{n,1}+1}{\sqrt n}\le t \right\}\to \sum_{\ell\in\Z}\left( 1-2(\ell t)^2 \right)e^{-(\ell t)^2},
	\qquad n\to\infty,
	\]
	by Theorem~\ref{thm:range_weaklimit}.
	This shows that the distribution of the range of $1$-watermelons without wall weakly converges to the limiting distribution
	of the height of $1$-watermelons with wall (see \cite{watermelons:withwall}).
\end{remark}

\bibliographystyle{plain}
\bibliography{pmelons}

\end{document}